\documentclass{amsart}
\usepackage{amssymb,amsmath,amsrefs}
\usepackage[colorlinks=true]{hyperref}
\usepackage[T1]{fontenc}
\hypersetup{urlcolor=blue, citecolor=green}
\usepackage[dvips]{graphicx}
\usepackage{color}
\usepackage{svg}
\usepackage{subcaption}
\usepackage{float}
\usepackage{shellesc}
\usepackage{tabularx}
\usepackage{enumitem}

\graphicspath{{Imagens/}}
\theoremstyle{plain}

\newtheorem*{conj*}{Conjecture}
\newtheorem*{cor*}{Corollary}
\newtheorem{theorem}{Theorem}[section]

\newtheorem{proposition}[theorem]{Proposition}
\newtheorem{corollary}[theorem]{Corollary}
\newtheorem{lemma}[theorem]{Lemma}

\newtheorem{theoremm}{Theorem}

\theoremstyle{definition}
\newtheorem*{def*}{Definition}
\newtheorem{remark}[theorem]{Remark}

\newtheorem{definition}[theorem]{Definition}

\newcommand{\SR}{{\mathcal R}}

\newcommand{\SU}{{\mathcal U}}

\newcommand{\com}{\operatorname{Comp}}
\newcommand{\Z}{\mathbb{Z}}
\newcommand{\N}{\mathbb{N}}
\newcommand{\R}{\mathbb{R}}
\newcommand{\eps}{\varepsilon}

\newcommand{\dist}{\operatorname{\textit{d}}}

\newcommand{\const}{\operatorname{const}}

\newcommand{\diff}{\operatorname{Diff}}

\newcommand{\diam}{\operatorname{diam}}

\DeclareMathOperator{\image}{Im}

\makeatletter
\newcommand{\tpitchfork}{%
  \vbox{
    \baselineskip\z@skip
    \lineskip-.52ex
    \lineskiplimit\maxdimen
    \m@th
    \ialign{##\crcr\hidewidth\smash{$-$}\hidewidth\crcr$\pitchfork$\crcr}
  }%
}
\makeatother

\author{Alfonso Artigue}
\author{Bernardo Carvalho}
\author{José Cueto}

\title[$C^1$-generic continuum-wise expansive surface diffeomorphisms]{$C^1$-generic continuum-wise expansive surface diffeomorphisms}
\begin{document}

\renewcommand{\thefootnote}{}

\footnote{2020 \emph{Mathematics Subject Classification}: Primary 37D20; Secondary 37C05.}

\footnote{\emph{Key words and phrases}: Continuum-wise, Expansiveness, generic, diffeomorphisms}

\renewcommand{\thefootnote}{\arabic{footnote}}
\setcounter{footnote}{0}

\begin{abstract}
We exhibit a local residual set of surface $C^1$ diffeomorphisms that are continuum-wise expansive but are not expansive. We also exhibit an open and dense set of surface $C^1$ diffeomorphisms where expansiveness implies being Anosov.
\end{abstract}

\maketitle

\section{Introduction}

The concept of hyperbolicity is central in the theory of differentiable dynamical systems. It appeared in the 1960’s with the works of Dmitry Anosov \cite{An} and Stephen Smale \cite{S} and has been the main subject of research among many mathematicians since then. A general goal of the theory is to obtain a classification of most $C^1$ diffeomorphisms, where by most we mean generic (diffeomorphisms on a residual subset of the set of all diffeomorphisms $\diff^1(M)$ of a closed smooth manifold endowed with the $C^1$ topology). For some time it was expected that the hyperbolic systems were most of the $C^1$ diffeomorphisms, but this was proved to be false by R. Abraham and Stephen Smale \cite{AS} and also by Carl Simon \cite{Si}. They proved that hyperbolicity is not $C^1$ dense in the set of all diffeomorphisms of manifolds with dimension greater than two. In the case of surface diffeomorphisms, Sheldon Newhouse proved the non-density of hyperbolicity \cite{newhouse} and exhibited open sets of diffeomorphisms where a generic system possesses an infinite number of coexisting sinks \cite{newhouse2}. 

The works of Ricardo Mañé were fundamental for the theory culminating with the proof of the Stability Conjecture \cite{manhe} proposed by Jacob Palis and Stephen Smale \cite{palissmale} obtaining that hyperbolicity is a necessary feature of stable systems. He proved a dichotomy \cite{manhe4} which states that a $C^1$-generic surface diffeomorphism is either hyperbolic, i.e. Axiom A and satisfies the no-cycle condition, or exhibits the Newhouse phenomenon, admitting either infinitely many sinks or infinitely many sources. Mañé also discussed the relation between hyperbolicity and expansiveness, proving that the $C^1$ interior of the set of expansive diffeomorphisms equals the set of quasi-Anosov diffeomorphisms \cite{manhe2}. In the case of surface diffeomorphisms, being quasi-Anosov is equivalent to being Anosov \cite{manhe}, so this result implies that there is no open sets of expansive surface diffeomorphisms outside the set of Anosov diffeomorphisms. This highlights expansiveness as a fundamental property of hyperbolic diffeomorphisms and raises the question of whether there exist expansive diffeomorphisms beyond the closure of the set of (quasi-)Anosov diffeomorphisms.

In words, expansiveness means that the iterates of distinct points of the space must eventually separate from each other some uniform distance. This ensures strong relations between the dynamics of the system and the topology of the space. Koichi Hiraide \cite{hiraide} and Jorge Lewowicz \cite{lewowicz} explored this idea and obtained a classification result of expansive homeomorphisms of surfaces: they are either topologically conjugate to Anosov diffeomorphisms of the Torus $\mathbb{T}^2$, or are Pseudo-Anosov homeomorphisms of higher genus surfaces. However, a classification of the set of expansive diffeomorphisms with respect to hyperbolicity is still unknown.

Alexander Arbieto proved in \cite{arbieto1} that $C^1$-generic expansive diffeomorphisms are hyperbolic (Axiom A and satisfy the no-cycle condition) adapting the proof of Shaobo Gan and Dawei Yang in the case of generic expansive homoclinic classes \cite{GY}. This means that the set of expansive diffeomorphisms that are not hyperbolic is a meager set (complement of a residual) but it is still an open problem to determine whether $C^1$-generic expansive diffeomorphisms are necessarily quasi-Anosov, that is, whether the set of expansive and hyperbolic diffeomorphisms that are not quasi-Anosov is a meager set. We will discuss this problem in more detail in a subsequent work but we obtain the following result for surface diffeomorphisms.

\begin{theoremm}\label{theorem:1}
Let $M$ be a closed surface and $\mathcal{A} \subset \mathrm{Diff}^1(M)$ be the set of Anosov diffeomorphisms. Then there exists an open and dense subset $\mathcal{B} \subset \mathrm{Diff}^1(M) \setminus \overline{\mathcal{A}}$ such that no diffeomorphism in $\mathcal{B}$ is expansive.
\end{theoremm}

The main result of this article solves an analogous problem for the generalization of expansiveness called \emph{continuum-wise expansiveness}. This property was introduced by Hisao Kato in \cite{kato1}, \cite{kato2} and systems satisfying it may exhibit even a Cantor set of distinct initial conditions whose orbits do not separate, as in the case of the pseudo-Anosov diffeomorphism of $\mathbb{S}^2$ \cite{Walters}, but the diameter of the iterates of non-trivial continua (compact and connected subsets) necessarily increase a uniform amount. We are far from obtaining a classification of cw-expansive homeomorphisms and new examples have been appearing in recent years \cite{AAV}, \cite{Artanom}, \cite{ArtigueDend}, \cite{Artrobn}, \cite{APV}, \cite{CC} (see also the articles about cw-hyperbolicity \cite{ACS}, \cite{ACCV3}, \cite{CR}, \cite{COR}).

Regarding cw-expansive diffeomorphisms, Kazuhiro Sakai \cite{sakai0} proved that the $C^1$-interior of the set of cw-expansive diffeomorphisms coincides with the set of quasi-Anosov diffeomorphisms (which equals the $C^1$-interior of the set of expansive diffeomorphisms). Also, Manseob Lee \cite{lee} proved that $C^1$-generic cw-expansive diffeomorphisms are Axiom A and satisfy the no-cycle condition. We consider the problem of determining whether $C^1$-generic cw-expansive diffeomorphisms are necessarily quasi-Anosov (Anosov in the case of surface diffeomorphisms) and prove, in this article, that this is not the case. We exhibit a local residual set of surface $C^1$ diffeomorphisms that are continuum-wise expansive but are not expansive. This is the content of the main result of this article.

\begin{theoremm}\label{theorem:main_cw_exp1}
There exists a local residual set $\SR$ of $C^1$ diffeomorphisms of the bitorus where each element is $cw$-expansive and not expansive.
\end{theoremm}

The construction of this local residual set starts with the example of a 2-expansive (and not expansive) homeomorphism of the bitorus introduced in \cite{APV}, which is hyperbolic and has a pair of stable/unstable foliations with tangencies, then for each $n\in\N$ we consider sufficiently close diffeomorphisms where all curves of tangencies between the respective pairs of stable/unstable foliations have diameter less than $1/n$, note that when a diffeomorphism satisfies this for every $n\in\N$ it is necessarily cw-expansive, and then we discuss how to perform $C^1$ perturbations to break all possible curves of tangencies into pieces with arbitrarily small diameter. This idea is based in techniques provided by Alfonso Artigue in \cite{artigue2} where a $C^0$-residual subset of almost cw-expansive homeomorphisms on surfaces is constructed.

\section{\texorpdfstring{$C^1$-generic expansive surface diffeomorphisms}{C1-generic expansive surface diffeomorphisms}}

In this section, we begin with the precise definitions we will need during this article and then we prove Theorem \ref{theorem:1}. In what follows, $(X,d)$ is a compact metric space and $f\colon X\to X$ is a homeomorphism.

\begin{definition}[Local stable/unstable sets]\label{loc}
For each $x\in X$ and $c>0$, let 
$$W^s_{c}(x):=\{y\in X; \,\, d(f^k(y),f^k(x))\leq c \,\,\,\, \textrm{for every} \,\,\,\, k\geq 0\}$$
be the \emph{c-stable set} of $x$ and
$$W^u_{c}(x):=\{y\in X; \,\, d(f^k(y),f^k(x))\leq c \,\,\,\, \textrm{for every} \,\,\,\, k\leq 0\}$$
be the \emph{c-unstable set} of $x$. We consider the \emph{stable set} of $x\in X$ as the set 
$$W^s(x):=\{y\in X; \,\, d(f^k(y),f^k(x))\to0 \,\,\,\, \textrm{when} \,\,\,\, k\to\infty\}$$
and the \emph{unstable set} of $x$ as the set 
$$W^u(x):=\{y\in X; \,\, d(f^k(y),f^k(x))\to0 \,\,\,\, \textrm{when} \,\,\,\, k\to-\infty\}.$$
\end{definition}

\begin{definition}[Expansiveness]
For each $x\in X$ and $c>0$ let $$\Gamma_{c}(x)=W^u_{c}(x)\cap W^s_{c}(x)$$ be the \emph{dynamical ball} of $x$ with radius $c$. We say that $f$ is \emph{expansive} if there exists $c>0$ such that $$\Gamma_c(x)=\{x\} \,\,\,\,\,\, \text{for every} \,\,\,\,\,\, x\in X.$$ A constant $c$ satisfying the above definition is called an expansivity constant of $f$.
\end{definition}

Expansiveness was introduced by Utz in \cite{Utz}, explored by many authors since then, and gives information on topological and statistical aspects of chaotic systems. Many features of expansive homeomorphisms are known to these days and we invite the authors to read the monograph \cite{AH} for more information.

%\begin{definition}\label{bola_dinamica1}
%For a homeomorphism $g\colon S \to S$, $x\in S$, and $\varepsilon > 0$, we define the dynamical ball of radius $\varepsilon$ centered on $x$ as
%$$\Gamma^g_\varepsilon(x)=\{y\in S\colon\ d(g^n(x),g^n(y))\leq\varepsilon, \mathrm{\  for\  all\  } n\in\mathbb{Z} \}.$$ We say that $g$ is \emph{expansive} if there exists $\eps>0$ such that 
%$$\Gamma^g_\varepsilon(x)=\{x\} \,\,\,\,\,\, \text{for every} \,\,\,\,\,\, x\in X.$$ 
%\end{definition}

%\begin{definition}[\cite{manhe4}]\label{definition:lower_upper_semicontinuous}
%
%We say that a map \( \Gamma : X \to \mathcal{H}(M) \) is \emph{lower semicontinuous} in a point \( x \in X \) if for every open set \( U \) that intersects \( \Gamma(x) \) there exists a neighborhood \( \mathcal{U} \) of \( x \) in \( X \) such that if \( y \in \mathcal{U} \) then \( \Gamma(y) \cap U \neq \emptyset \). We say that \( \Gamma \) is \emph{upper semicontinuous} in \( x \) if for every neighborhood \( U \) of \( \Gamma(x) \) there exists a neighborhood \( \mathcal{U} \) of \( x \) in \( X \) such that \( y \in \mathcal{U} \) implies \( \Gamma(y) \subset U \). If $\Gamma$ is both upper and lower semi-continuous at a point $x_0\in X$, then we simply say that $\Gamma$ is continuous at $x_0\in X$.
%\end{definition}

\begin{definition}[Periodic and non-wandering points]
A periodic point is a point $p\in X$ such that there is $k\in\N$ satisfying $f^k(x)=x$. The period of $p$ is defined as the smallest number $k\in\N$ satisfying $f^k(x)=x$. The set of periodic points of $f$ is denoted by $Per(f)$. A point \( z\in X \) is called \emph{non-wandering for \( f \)} if for every neighborhood \( U \) of \( z \) there is \( n \geq 1 \) such that \( f^n(U) \) intersects \( U \). The set of non-wandering points of $f$ is denoted by \( \Omega(f) \). It contains the set \( \text{Per}(f) \) of periodic points of $f$, as well as the \( \alpha \)-limit set and the \( \omega \)-limit set of every orbit. 
\end{definition}

\begin{definition}[Sinks and Sources]
Let $M$ be a closed smooth manifold, $TM$ denotes its tangent bundle, $f\colon M\to M$ be a $C^1$ diffeomorphism, and $Df\colon TM\to TM$ be its derivative map. We say that $p\in Per(f)$ is a \textit{sink} if each eigenvalue of the derivative $Df^\nu(p)$ has absolute value less than one, where $\nu$ is the period of $p$. Similary, we say that $p$ is a \textit{source} if each eingenvalue of the derivative $Df^\nu(p)$ has absolute value great than one.
\end{definition}

Hyperbolic diffeomorphisms are defined as being Axiom A and satisfying the no-cycle condition.

\begin{definition}[Axiom A diffeomorphisms]
Axiom A diffeomorphisms are the ones whose non-wandering set is hyperbolic and is the closure of the set of periodic points. A hyperbolic set $\Lambda\subset M$ is a compact and invariant set $f(\Lambda)=\Lambda$ such that there is an invariant splitting of the tangent bundle $T\Lambda=E^s\oplus E^u$ and constants $c>0$ and $\lambda\in(0,1)$ such that 
$$\|Df^k_{|E^s}\|\leq c\lambda^k \,\,\, \textrm{and}
      \,\,\, \|Df^{-k}_{|E^u}\|\leq c\lambda^k \,\,\,\,\,\, \text{for every} \,\,\,\,\,\, k\in\N,$$
where $\|.\|$ is a Riemannian norm on $TM$.
\end{definition}

It was proved by Smale in \cite{S} that Axiom A diffeomorphisms satisfy the Spectral Decomposition Theorem which ensures that the non-wandering set is the union of a finite number of disjoint, compact, invariant, and transitive sets $\Delta_1, \dots, \Delta_n$ called basic sets. 

\begin{definition}[No-cycle condition]\label{cycle_condition}
We say that $f$ has a \emph{cycle} between these basic sets if there exist $\Delta_{i_1}, \dots, \Delta_{i_{k-1}}, \Delta_{i_k}=\Delta_{i_1}$ such that 
$$W^u(\Delta_{i_j}) \cap W^s(\Delta_{i_{j+1}}) \neq \emptyset, \quad \forall \,\, 1 \leq j \leq k-1.$$ When this intersection holds, we write $\Delta_{i_j}\gg\Delta_{i_{j+1}}$. We say that an Axiom A diffeomorphism satisfies the no-cycle condition if there are no cycles between all basic sets of $f$. 
\end{definition}

The Axiom A diffeomorphisms satisfying the no-cycle condition are characterized as being star, that is, they form exactly the $C^1$ interior of the set of diffeomorphisms whose periodic orbits are all hyperbolic denoted by $\mathcal{F}^1(M)$ (see \cite{hayashi}). The $C^1$ interior of the set of expansive diffeomorphisms is characterized as the set of quasi-Anosov diffeomorphisms (see \cite{lewowicz_cerminara}, \cite{manhe2}, and \cite{Ma3}).

\begin{definition}[Quasi-Anosov diffeomorphism]
We say that a diffeomorphism $f\colon M\to M$ is a \emph{quasi-Anosov diffeomorphism} if the sequence $(\|Df^n(x)(v)\|)_{n\in\Z}$ is not bounded for every $x\in M$ and $v\in T_xM\setminus\{0\}$.
\end{definition}
 
Equivalent definitions were obtained in \cite{manhe2} and quasi-Anosov diffeomorphisms can be defined as the hyperbolic diffeomorphisms satisfying the quasi-transversality condition.

\begin{definition}[Quasi-transversality condition]
We say that a hyperbolic diffeomorphism satisfies the \emph{quasi-transversality} condition if $$T_x W^s(x) \cap T_x W^u(x) = \{0\} \,\,\,\,\,\, \text{for every} \,\,\,\,\,\, x\in M.$$
\end{definition}

This means that all stable and unstable manifolds do not have common tangent directions. We recall the definition of the strong-transversality condition and that hyperbolic diffeomorphisms satisfying the strong-transversality condition are necessarily structurally stable \cite{robinson_2}.

\begin{definition}[Strong-transversality condition]
We say that a hyperbolic diffeomorphism $f\colon M\to M$ satisfies the \emph{strong-transversality} condition if the stable manifold $W^s(x)$ and the unstable manifold $W^u(x)$ are transversal for every $x \in M$. This means that at every point $x \in M$ we have $$T_x W^s(x) \oplus T_x W^u(x) = T_x M.$$
\end{definition}

On surfaces, strong-transversality and quasi-transversality are equivalent, which ensures that the classes of quasi-Anosov and Anosov diffeomorpshisms are the same. On manifolds of dimension greater than 2, it is easy to see the difference between these two definitions, such as when there is a point where stable/unstable manifolds are curves with distinct tangent directions on a manifold of dimension 3. In this case, these curves satisfy the quasi-transversality condition at this point but do not satisfy the strong-transversality condition.

Thus, the class of quasi-Anosov diffeomorphisms is distinct from both classes of hyperbolic and Anosov diffeomorphisms. Indeed, John Franks and Clark Robinson \cite{FR} exhibited examples of a quasi-Anosov diffeomorphisms that are not Anosov (on manifolds with dimension greater than 2) and also there are examples of hyperbolic diffeomorphisms that are not quasi-Anosov such as the example of the 2-expansive diffeomorphism in \cite{APV} on the bitorus.

Being Anosov, quasi-Anosov, or hyperbolic are robust features of $C^1$ diffeomorphisms. Expansiveness, however, is not $C^1$-robust. Pseudo-Anosov diffeomorphisms on surfaces of genus greater than 1 have non-hyperbolic fixed points which can be perturbed to sinks/sources, contradicting expansiveness since expansive systems do not have stable/unstable points. Also, examples of topologically hyperbolic diffeomorphisms on the Torus, which satisfy both expansiveness and shadowing, but are not hyperbolic can be found in \cite{lewowicz2}. $C^1$-robust expansiveness only happens for quasi-Anosov diffeomorphisms, so on surfaces for Anosov diffeomorphisms. The $C^1$-generic surface diffeomorphisms can be classified as in the following result of Ricardo Mañé.

\begin{theorem}[\cite{manhe4} Corollary II]\label{corollary:manhe1}
If $M$ is a closed surface, then there exists a residual subset $\mathcal{O}\subset~\diff^1(M)$ such that every $f\in\mathcal{O}$ satisfies one of the following properties:
\begin{enumerate}
\item $f$ has infinitely many sinks,
\item $f$ has infinitely many sources or
\item $f$ satisfies Axiom A and the no-cycle condition.
\end{enumerate}
\end{theorem}

Using this result we can prove Theorem \ref{theorem:1}, that is, diffeomorphisms in an open and dense set in $\diff^1(M)$, where $M$ is a closed surface, which are expansive are necessarily Anosov.
%For the moment, we continue to assume that $\dim(M) = 2$. 

\begin{proof}[Proof of Theorem \ref{theorem:1}]
Let $\mathcal{R}_1\subset\diff^1(M)$ be the set of diffeomorphisms satisfying either item (1) or (2) of Theorem \ref{corollary:manhe1}, that is, that have either infinitely many sinks or infinitely many sources, and recall that $\mathcal{F}^1(M)$, the set of star diffeomorphisms, equals the set of hyperbolic diffeomorphisms, which satisfy item (3). Thus, we have that the residual set obtained in the above theorem satisfies \mbox{$\mathcal{O}=\mathcal{R}_1\cup\mathcal{F}^1(M)$}. Consider the set $$\mathcal{C}=\{ f \in \diff^1(M) : f \text{ has at least one sink or one source} \}.$$ This is an open subset of $\diff^1(M)$ where no element is expansive since the existence of sinks or sources contradict expansiveness. Let $$\mathcal{B} = \mathcal{C}\cup \left( \mathcal{F}^1(M) \setminus \overline{\mathcal{A}}\right)$$ and note that $\mathcal{B}$ is an open subset of $\diff^1(M)$. We prove that $\mathcal{B}$ is dense in $\diff^1(M)\setminus\overline{\mathcal{A}}$. Indeed, each $f\in\diff^1(M)\setminus\overline{\mathcal{A}}$ is accumulated by diffeomoprhisms of $\mathcal{O}$ since $\mathcal{O}$ is a residual subset, and hence dense, in $\diff^1(M)$. Since $\mathcal{O}=\mathcal{R}_1\cup\mathcal{F}^1(M)$, it follows that $f$ is either accumulated by diffeomorphisms in $\mathcal{R}_1$ or by diffeomorphisms in $\mathcal{F}^1(M) \setminus \overline{\mathcal{A}}$. In the first case, $f$ is accumulated by diffeomorphisms in $\mathcal{C}$ since $\mathcal{R}_1\subset \mathcal{C}$. This proves the density of $\mathcal{B}$ in $\diff^1(M)\setminus\overline{\mathcal{A}}$. The fact that no $f\in\mathcal{B}$ is expansive follows from the fact that every $f \in \mathcal{B}$ satisfies $\Omega(f) \neq M$. Indeed, diffeomorphisms admitting sinks or sources are not transitive and hyperbolic diffeomorphisms that are not Anosov are also not transitive. Also, expansive surface diffeomorphisms are pseudo-Anosov, which are transitive (see \cite[Theorem 14.19]{farb}). This completes the proof.
%It is obvious that $\mathcal{C}$ and $\mathcal{F}^1\setminus\overline{\mathcal{A}}$  are open sets. It is also known that expansive and \( cw \)-expansive homeomorphisms on manifolds do not have stable points.  
%Therefore the set \( \mathcal{C}\), of diffeomorphism that has at least one sink or one source is open, and its elements are not \( cw \)-expansive.   
%\noindent $f\in\overline{\mathcal{R}}_1$ then $f\in\overline{\mathcal{C}}$. If  $f\in\overline{\mathcal{F}^1}\setminus\overline{\mathcal{C}}$ then $f\in\overline{\mathcal{F}^1}\setminus\overline{\mathcal{A}}.$
%\noindent Finally, statement \ref{item:exp_then_anosov3} follows from the fact that, by definition of $\mathcal{B}$, every $f \in \mathcal{B}$ satisfies $\Omega(f) \neq M$, for more detail on this later see \cite[Theorem 14.19]{farb}.
\end{proof}

\begin{corollary}\label{generic_exp_implies_anosov}
There is an open and dense subset $\mathcal{U}\subset\diff^1(M)$ such that if a diffeomorphism  in $\mathcal{U}$ is expansive, then it is Anosov.
\end{corollary}
\begin{proof}
Considering the open set $\mathcal{B}$ defined in Theorem~\ref{theorem:1} above, and noting that $\mathcal{B}$ is dense in $\diff^1(M) \setminus \overline{\mathcal{A}}$, we conclude that the set $\mathcal{U} = \mathcal{A} \cup \mathcal{B}$ is open and dense. Moreover, by the natural properties of these sets, every diffeomorphism is either non-expansive or Anosov.
\end{proof}

Thus, the set of expansive surface diffeomorphisms that are not Anosov is nowhere dense.
In Figure~\ref{fig:decomposition} we illustrate the open and dense subset 
$\mathcal{C}\cup\mathcal{F}^1(M)\subset\diff^1(M)$. 
As we will see in the following section, the set of cw-expansive diffeomorphisms contains a local residual subset of hyperbolic diffeomorphisms disjoint from the set of Anosov diffeomorphisms.

\begin{figure}[!htb]
\centering
\includegraphics[width=0.9\linewidth]{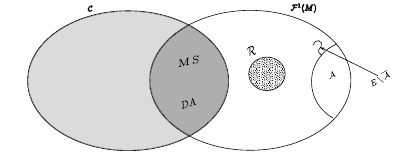}
\caption{In this picture $E$ denotes the set of expansive diffeomorphisms, $MS$ denotes the set of Morse-Smale diffeomorphisms, and $DA$ denotes the Derived-from-Anosov diffeomorphism. 
Also we show the local residual set $\SR$ of Theorem \ref{theorem:main_cw_exp1}.}
\label{fig:decomposition}
\end{figure}

\section{Generic cw-expansiveness}\label{section:residual}

We begin this section recalling the construction of the of the $2$-expansive diffeomorphism $f_0$ of the bitorus. Then we write the set of cw-expansive diffeomorphisms on a sufficiently small $C^1$-neighborhood of $f_0$ as a residual subset $\mathcal{R}=\bigcap_{n\in\N}F_n$. The last part is devoted to prove that indeed $F_n$ is open and dense in this neighborhood for every $n\in\N$. The difficult part of the proof of the density of $F_n$ is solved by the perturbation lemmas we perform at the end of this section.

%Here we construct on the bitorus a diffeomorphism  $f_0$ that is $2$-expansive, and by means of perturbation lemmas, we obtain a residual subset of cw-expansive diffeomorphisms that are not expansive. 
%The strategy is to consider a small neighborhood of the diffeomorphism $f_0$ such that, for any diffeomorphism in this neighborhood, the tangential intersections between its stable and unstable leaves remain contained in a perturbation domain near that of the  diffeomorphism $f_0$. To obtain a residual subset in this neighborhood, we take the intersection of a countable family of open and dense subsets. This family consists of diffeomorphisms whose  dynamical balls have components of arbitrarily small diameter, so that any diffeomorphism in the residual subset has all its dynamical balls totally disconnected. To achieve this, we use perturbation lemmas that break a dynamical ball into sufficiently small components.

%\subsection{Derived-from-Anosov diffeomorphism}

\subsection{Construction of DA on the Bitorus}
\label{secDefDA}
Consider \( M_1 \) and \( M_2 \) two copies of the torus \(\mathbb{T}^2 \) and \mbox{$ C^\infty$-diffeomorphisms} \( f_i: M_i \to M_i \), \( i = 1,2 \), such that:
\begin{enumerate}
    \item \( f_1 \) is derived-from-Anosov  (see for example \cite{robinson_livro} Section 7.8 for a construction of
such a map),
    \item \( f_2 \) is conjugate to \( f_1^{-1} \), and
    \item \( f_i \) has a fixed point \( p_i \), where \( p_1 \) is a source and \( p_2 \) is a sink.
\end{enumerate}
Also consider local charts \( \varphi_i: D \to M_i \), where $$ D = \{x \in \mathbb{R}^2 : \|x\| \leq 8\}, $$ such that:
\begin{enumerate}
    \item \( \varphi_i(0) = p_i \),
    \item  the pull back of the stable (unstable) foliation, $W^s(W^u)$, by \( \varphi_1 \) (\( \varphi_2 \)) is the horizontal (vertical) foliation on \( D \), and
    \item \( \varphi_1^{-1} \circ f_1 \circ \varphi_1(x) = \varphi_2^{-1} \circ f_2^{-1} \circ \varphi_2(x) = 2x \) for all \( x \in D \).
\end{enumerate}
Let \( A\subset D \) be the annulus $$A=\{x \in \mathbb{R}^2 : 1/8 \leq \|x\| \leq 8\}$$ and consider the inversion diffeomorphism \( \psi\colon A\to A \) given by \( \psi(x) = x/\|x\|^2 \). Denote by \( \hat{D}\subset D \) the open disk $$\hat{D}=\{x \in \mathbb{R}^2 : \|x\| < 1/8\}$$ and consider the equivalence relation $\sim$ on \( [M_1 \setminus \varphi_1(\hat{D})] \cup [M_2 \setminus \varphi_2(\hat{D})] \) generated by:
$$\varphi_1(x) \sim \varphi_2 \circ \psi(x) \,\,\,\,\,\, \text{for every} \,\,\,\,\,\, x \in A.$$
Consider the surface
\[
M = \frac{[M_1 \setminus \varphi_1(\hat{D})] \cup [M_2 \setminus \varphi_2(\hat{D})]}{\sim},
\]
which has genus two, and the quotient topology on \( M \).
Denote by $[x]$ the equivalence class of \( x \) and define the \( C^\infty \)~diffeomorphism \( f_0\colon M \to M \) by
\begin{equation}\label{c2_expanive}
f_0([x]) =
\begin{cases}
[f_1(x)] & \text{if } x \in M_1 \setminus \varphi_1(\hat{D}) \\
[f_2^{-1}(x)] & \text{if } x \in M_2 \setminus \varphi_2(\hat{D}).
\end{cases}
\end{equation}
Note that $f_0$ is well defined. Indeed, let
\begin{equation}\label{definition:anulus}
\widetilde{A}:=(M_1\setminus \varphi_1(\hat{D})/\sim)\cap(M_2\setminus\varphi_2(\hat{D})/\sim)=\varphi_1(A)\cup\varphi_2(A)/\sim
\end{equation}
and note that if $z\in\widetilde{A}$, $z$ is the equivalence class of $q\in A_2=\varphi_2(A)$, and $q'\in A_1=\varphi_1(A)$, then $\psi\circ f_2^{-1}(q)=f_1(\psi(q))=f_1(q')$, so $f_2^{-1}(q)\sim f_1(q')$.

By construction we know that \( f_0 \) is Axiom A and the non-wandering set consists of a hyperbolic repeller and a hyperbolic attractor. There are two one-dimensional foliations on $A$, the restrictions of $W^s$ on $M_1$ and $W^u$ on $M_2$. We will denote these also by $W^s$ and $W^u.$ The foliation $W^s$ consists of horizontal lines and the foliation $W^u$ is a more complicated ``dipole'' foliation as in Figure~\ref{fig:dipole}. 
\begin{figure}[!htb]
\centering
\includegraphics[width=0.8\linewidth]{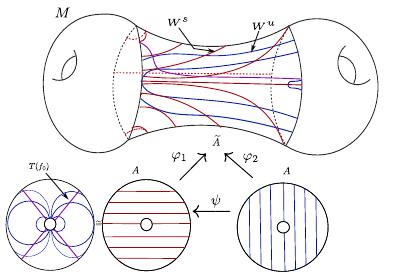}
\caption{\small Foliations $W^s$ and $W^u$ in the annulus $\widetilde{A}.$
}
\label{fig:dipole}
\end{figure}

\subsection{Cw-expansive diffeomorphisms}
We consider a $C^1$-neighborhood $\mathcal{U}_0$ of the 2-expansive diffeomorphism $f_0$ satisfying the following conditions:
\begin{itemize}
\item all diffeomorphisms $f\in\mathcal{U}_0$ have a non-wandering set formed by an attractor/reppeler pair which are hyperbolic,
\item each $f\in\mathcal{U}_0$ has a pair of one-dimensional stable/unstable foliations on the anulus $\widetilde A$,
\item there exists $c>0$ such that $W^s_c(x)$ and $W^u_c(x)$ are contained in the respective stable/unstable leaves of the foliations for every $x\in M$ and all diffeomorphisms $f\in\mathcal{U}_0$. 
\end{itemize}
We define the subset $\mathcal{R}\subset\mathcal{U}_0$ as $$\mathcal{R}=\bigcap_{n\in\mathbb{N}}F_n$$ where $F_n$ is defined as
\begin{equation}\label{definition_F_n}
F_n := \left\{ g \in \mathcal{U}_0 : \sup_{x \in M} \diam(\com_{x}\Gamma^g_c(x))<\frac{1}{n} \right\}.
\end{equation}
Here $\com_{x}A$ denotes the connected component of the set $A$ that contains $x$ and $\Gamma^g_c(x)$ denotes dynamical ball of $x$ with radius $c$ where the points are iterated by $g$.

%\begin{remark}Suppose $c>0$ is such that $W^s_c(x)\subset\Phi(x)$ where $\Phi$ is given by the Stable Manifold Theorem and similarly for $W^u_c(x)$, for all difeomorphism $f\in\mathcal{U}$. This cw-expansivity constant will be the same for every diffeomorphism $f$ in the residual subset $\mathcal{R}$.
%\end{remark}

\begin{proposition}\label{proposition:residual}
Every $f\in\mathcal{R}$ is cw-expansive.
\end{proposition}

\begin{proof}
If $f\in\mathcal{R}$, then $f\in F_n$ for every $n\in\mathbb{N}$. Thus,
$$\diam(\com_{x}\Gamma^f_c(x)) =0 \,\,\,\,\,\, \text{for every} \,\,\,\,\,\, x\in M.$$
This ensures that $\Gamma^f_c(x)$ is totally disconnected for every $x\in M$ from where cw-expansiveness follows.
\end{proof}

All that follows is devoted to prove that $F_n$ is an open subset of $\SU_0$ and dense in some $C^1$-neighborhood $\mathcal{U}\subset\mathcal{U}_0$ of $f_0$ for every $n\in\mathbb{N}$.

\begin{definition}[Inferior and superior limits]
Suppose that $\{X_n\}_{n\in\N}$ is a sequence of subsets of a space $M$. The set of all points $x$ in $M$ such that every open set containing $x$ intersects all but a finite number of the sets $X_n$ is called the \textit{limit inferior} of the sequence $\{X_n\}_{n\in\N}$ and is abbreviated ``$\liminf X_n$''. The set of all points $y$ in $M$ such that every open set containing $y$ intersects infinitely many sets $X_n$ is called the \textit{limit superior} of $\{X_n\}_{n\in\N}$ and is abbreviated ``$\limsup X_n$." If these two sets coincide (so that $\liminf X_n = L = \limsup X_n$), we say that $\{X_n\}_{n\in\N}$ is a \textit{convergent sequence} of sets and that $L$ is the \textit{limit} of $\{X_n\}_{n\in\N}$, which is abbreviated ``$L = \lim X_n$."
\end{definition}

\begin{theorem}[\cite{hocking} Theorem 2-101]\label{hocking_theorem_2_101} If $\{X_n\}_{n\in\N}$ is a sequence of connected sets in a compact Hausdorff space $S$, and if $\liminf X_n$ is not empty, then $\limsup X_n$ is connected.
\end{theorem}

\noindent We have the following results for a sequence of  homeomorphisms $(g_n)_{n\in\N}$. It will  be applied for $ C^1$-diffeomorphisms, but we give a more general proof in the $C^0$ setting.

\begin{lemma}\label{lem_bola_dinamica}
If $g_n\to g$ in the $C^0$ topology and $y_n\to z$ in $M$, then $$\limsup\Gamma^{g_n}_c(y_n)\subset\Gamma^g_c(z).$$
\end{lemma}

\begin{proof}
Let $u\in\limsup\Gamma^{g_n}_c(y_n)$. Take $w_n\in\Gamma^{g_n}_c(y_n)$ such that $w_n\to u$.
%Given $r>0$, by definition, there is an infinite subset $N\subset\mathbb{N}$ and points $(w_n)_{n\in N}$ such that 
%$$w_n\in\Gamma^{g_n}_c(y_n)\cap B_{r}(u)\ \text{ for all } n\in N.$$
Since $g_n\to g$ and $w_n\to u$, for each $m\in\Z$ we can take the limit 
$$\dist(g^m_n(w_n),g^m(u))\leq\dist(g^m_n(w_n),g^m(w_n))+\dist(g^m(w_n),g^m(u))\to 0.$$
Also
%, since $y_n\to z$ we can take the limit
%$$\dist(g^m_n(y_n),g^m(z))\to 0.$$
%Now since 
$$\dist(g^m(u),g^m(z))\leq
\dist(g^m(u),g^m_n(w_n))+
\dist(g^m_n(w_n),g^m_n(y_n))+
\dist(g^m_n(y_n),g^m(z))$$ 
thus in the limit when $n\to\infty$ %and $\varepsilon\to 0,$ 
we have 
$$\dist(g^m(u),g^m(z))\leq c$$
and this concludes the proof.
\end{proof}

\begin{lemma}\label{open}
The set $F_n$ is an open subset of $\mathcal{U}_0$ for each $n\in\mathbb{N}$.
\end{lemma}

\begin{proof}
We prove that the complement of $F_n$ is closed in $\mathcal{U}_0$ for every $n\in\N$. Assume that $(g_n)_{n\in\N}\subset \mathcal{U}_0\setminus F_{k_0}$ for some \mbox{$k_0\in\mathbb{N}$} and that \mbox{$g_n\to g\in\mathcal{U}_0$} when $n\to\infty$. We will prove that $g\notin F_{k_0}$.
Fix $n\in\N$.
From $g_n\notin F_{k_0}$ we obtain that 
$$\sup_{x \in M}\diam(\com_{x}\Gamma^{g_{n}}_c(x))\geq\frac{1}{k_0}.$$ 
Thus, we can consider a sequence $(x_l)_{l\in\N}\subset M$ such that 
$$\lim_{l\to\infty}\diam(\com_{x_l}\Gamma^{g_{n}}_c(x_l))\geq\frac{1}{k_0}.$$ 
As $M$ is compact, we can assume that 
$(x_l)_{l\in\N}$ converges to $y_n$ (notice that $(x_l)_{l\in\N}$ depends on $n$).
Also, we can assume that $$\lim_{l\to\infty}\com_{x_{l}}\Gamma^{g_{n}}_c(x_{l})=L_n$$ 
for some continuum $L_n\subset M$.
By the continuity of the diameter function we have 
$\diam(L_n)\geq\frac{1}{k_0}$.
%Thus, we can consider a subsequence $(g_{n_k})_{k\in\N}$ of %$(g_n)_{n\in\N}$ and a sequence of points $(x_{n_k})_{k\in\N}\subset M$ converging to $x\in M$ such that 
%$$\lim_{k\to\infty}\com_{x_{n_k}}\Gamma^{g_{n_k}}_c(x_{n_k})=L$$ 
%is a continuum containing $x$ with $\diam(L)\geq\frac{1}{k_0}$ (this limit is taken in the Hausdorff topology).
Since \mbox{\( L_n = \lim_{l\to\infty} \com_{x_{l}}\Gamma^{g_n}_c(x_{l}) \)}
and
$\com_{x_l}\Gamma^{g_n}_c(x_l)\subset\Gamma^{g_n}_c(x_l)$
we have
$$L_n\subset \Gamma^{g_n}_c(y_n).$$
Again, as $M$ is compact, we can assume that $y_n\to z\in M$.
By Lemma~\ref{lem_bola_dinamica}  we obtain
$$  L=\lim_{n\to\infty} L_n\subset \limsup_{n\to\infty} \Gamma^{g_n}_c(y_n)\subset\Gamma^g_c(z).$$
By Theorem~\ref{hocking_theorem_2_101} we have that $L$ is connected.
As $\diam(L)\geq 1/k_0$ we conclude that
%hence since $x\in\liminf\com_{x}\Gamma^{g_n}_c(x)\neq\emptyset$, it is because  ensures that the $\limsup$ of a sequence of connected sets is connected, provided that the $\liminf$ is nonempty, we have
%$$  L\subset \limsup %\com_{x}\Gamma^{g_n}_c(x)\subset\com_{x}\Gamma^g_c(x),$$
%Therefore
%$$ \frac{1}{k_0} \leq \sup_{x \in M} \diam(L) \leq \sup_{x \in M} %\diam(\limsup \com_{x}\Gamma^{g_n}_c(x)) \leq \sup_{x \in M} %\diam(\com_{x}\Gamma^g_c(x)), $$
%this is 
$g\notin F_{k_0}$, which completes the proof.
\end{proof}

%\begin{remark}
%We recall that  
%\(\mathcal{R} = \bigcap_{n \geq 1} F_n.\)  If we prove that $\mathcal{R}$ is dense in $\mathcal{U}$ (with a uniform cw-expansive constant $c$), it follows that each $F_n$ is dense in $\mathcal{U}$. Therefore, by Proposition~\ref{proposition:residual}, the subset of cw-expansive $C^1$-diffeomorphisms  contained in \(\mathcal{U} \subset F_n\), which lies inside $F_n$, is dense in $\mathcal{U}$.
%This provides one way to prove that cw-expansive diffeomorphisms are dense in $\mathcal{U}$.
%\end{remark}

%Throughout this work, however, we present a different approach: first, we will prove that each $F_n$ is dense in some neighborhood $\mathcal{U}$, and then we construct a residual set of cw-expansive diffeomorphisms. This form of proof is much more advantageous, as it is consistent with the established hierarchy.

%Theorem \ref{theorem:F_n_dense} proves the density of $F_n$ .

We denote by 
\[
\tpitchfork(f) = W^u(\Omega(f), f) \pitchfork W^s(\Omega(f), f)
\]  
the set of points of transverse intersection between the stable and unstable manifolds in $M$.  

Let $\mathcal{W}^1 \subset \mathcal{F}^1(M)$ denote an open neighborhood  of $f_0$ given by the $\Omega$-stability of $f_0$, such that for each $f$ close to $f_0$ we have $\Omega(f)\cap\widetilde{A}=\emptyset$ ($\widetilde{A}$ as defined in Equation \eqref{definition:anulus}) and
 
\[
\Omega(f) = \Delta_1(f) \cup \Delta_2(f),
\]
 where $\Delta_1(f) \gg \Delta_2(f)$ are basic sets (recall $\gg$ from Definition \ref{cycle_condition}).  

We now show that the set-valued map that associates to each $f$ the set of points in $\widetilde{A}$ which are not transverse intersections of stable and unstable manifolds is upper semicontinuous. 
%in the sense of Definition~\ref{definition:lower_upper_semicontinuous}.

\begin{lemma}\label{lemma:T(f)upper_semicontinuous}\label{subsection:transversality}
For all $f\in\mathcal{W}^1$ the set $T(f) := \widetilde{A} \setminus \tpitchfork(f)$ is closed and the map
\[
T \colon \mathcal{W}^1 \to \mathcal{H}(\widetilde{A}), 
\]
is upper semicontinuous.
\end{lemma}

\begin{proof}
Let $f \in \mathcal{W}^1$ and $K \subset \widetilde{A}$ be a compact set such that 
$T(f) \cap K = \emptyset$. Assume that $\Delta_1(f) \gg \Delta_2(f)$.  
By the $\Omega$-stability we have that for every neighborhood $U'$ of $\Omega(f)$ there exists a neighborhood 
$\mathcal{U}'$ of $f$ in $\mathcal{W}^1$ such that, for all $g \in \mathcal{U}'$,
\[
\Delta_1(g) \cup \Delta_2(g) = \Omega(g) \subset U'.
\]
Since for all $k\geq0$
we have
\begin{eqnarray*}
f^{-k}\big(W^s_\varepsilon(\Delta_1(f),f)\big) &\subset& f^{-k-1}\big(W^s_\varepsilon(\Delta_1(f),f)\big) \,\,\,\,\,\, \text{and}\\
f^{k}\big(W^u_\varepsilon(\Delta_2(f),f)\big) &\subset& f^{k+1}\big(W^u_\varepsilon(\Delta_2(f),f)\big),
\end{eqnarray*} 
we can choose $n > 0$ sufficiently large so that  
\[
f^{-n}\big(W^s_\varepsilon(\Delta_1(f),f)\big) \quad \text{and} \quad  
f^{n}\big(W^u_\varepsilon(\Delta_2(f),f)\big)
\]
cover $\widetilde{A}$. 
Thus the  set $\tpitchfork(f)\cap\widetilde{A}$ is open and its complement $T(f)$ is closed.

Take $y\in K$ and take $U$ a closed neighborhood of $y$ contained in $\tpitchfork(f)\cap\widetilde{A}$. 
By the continuity of stable and unstable manifolds with respect to $f$ in the $C^1$-topology, there exists a neighborhood $\mathcal{U''}\subset\mathcal{U}'$ of $f$ such that
\[
U \subset g^{-n}\big(W^s_\varepsilon(\Delta_1(g),g)\big) \pitchfork g^{n}\big(W^u_\varepsilon(\Delta_2(g),g)\big),
\]
for all $g\in\mathcal{U}''$. Hence, for all $y'\in K$ there exist $\mathcal{U}'_{y'}$ neighborhood of $f$ and $U_{y'}$ neighborhood of  $y'$ such that $$U_{y'}\subset\tpitchfork(g)\cap\widetilde{A}$$ for all $g\in\mathcal{U}_{y'}\subset\mathcal{U}'$. By the compactness of $K$ we have that there exists a finite cover such that $K\subset\cup_{j=1}^{m} U_{y'_j}$, and considering $\mathcal{U}=\cap_{j=1}^{m}\mathcal{U}_{y'_j}\subset\mathcal{U}'$ we have $$K\subset\tpitchfork(g)\cap\widetilde{A}$$ for all $g\in\mathcal{U}.$ This proves the semicontinuity of the map $T$.
\end{proof}

\subsection{Perturbation Lemmas}
In what follows, we explain that it is possible to construct a perturbation within a foliated rectangle in the plane such that the resulting leaves are $C^1$-close to the original ones, but $C^2$-well-separated, as in the Figure~\ref{fig:h}.
Let $\varepsilon\in(0,1)$, $(\theta_0,\rho_0)\in\mathbb{R}^2$, $m\in\N$, $\delta_0\in(0,1)$, $\delta_1>0$, consider the rectangle
$$R_*=[\theta_0,\theta_0+\delta_0/m]\times[\rho_0,\rho_0+\delta_1],$$ and define the functions $\mu_{\theta_0},\mu_{\rho_0}\colon\mathbb{R}\to\mathbb{R}$ by
\begin{align*}
    \mu_{\theta_0}(\theta) &= \frac{\delta}{m} \cdot \varphi\left( \frac{m(\theta - \theta_0)}{\delta_0} \right), \\
    \mu_{\rho_0}(\rho) &= \varphi\left( \frac{\rho - {\rho_0}}{\delta_1} \right),
\end{align*}
where $$\delta=\frac{\min\{\delta_0,\delta_1\}\varepsilon/2}{\max^2_{t\in\mathbb{R}}\{ |\varphi(t)|,|\varphi'(t)|\}}$$
and $\varphi\colon\mathbb{R}\to[0,1]$ is a smooth bump function  that is zero for values less than zero and for values greater than 1, as in the Figure~\ref{fig:varphi}.

\begin{figure}[!htb]
\centering
\includegraphics[width=1\linewidth]{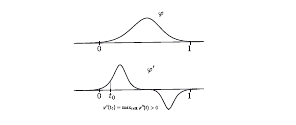}
\caption{\small Bump function $\varphi$ and its derivative.}
\label{fig:varphi}
\end{figure}

\begin{remark}
Under the above conditions, if $(\theta,\rho)\in\!R_*$, then
%computing the derivatives of $\varphi$ and taking the absolute value, we have 

\begin{eqnarray*}
&\Big|\frac{\delta}{m}\varphi\Big(\frac{m(\theta-\theta_0)}{\delta_0}\Big)\varphi\Big(\frac{\rho-{\rho_0}}{\delta_1}\Big)\Big|\leq\varepsilon/2, &\\
&
\Big|\frac{\delta}{\delta_0}\varphi'\Big(\frac{m(\theta-\theta_0)}{\delta_0}\Big)\varphi\Big(\frac{\rho-{\rho_0}}{\delta_1}\Big)\Big|\leq\varepsilon/2, &\\
&\Big|\frac{\delta}{m\delta_1}\varphi\Big(\frac{m(\theta-\theta_0)}{\delta_0}\Big)\varphi'\Big(\frac{\rho-{\rho_0}}{\delta_1}\Big)\Big|\leq\varepsilon/2. &
\end{eqnarray*}
Hence
\begin{equation}\label{inequality_epsilon}
\left.\begin{array}{ccc}
|\mu_{\theta_0}(\theta)\mu_{\rho_0}(\rho)|\leq\varepsilon/2, &
|\mu'_{\theta_0}(\theta)\mu_{\rho_0}(\rho)|\leq\varepsilon/2,&
|\mu_{\theta_0}(\theta)\mu_{\rho_0}'(\rho)|\leq\varepsilon/2,
\end{array}\right.
\end{equation}
for all $(\theta,\rho)\in R_*$.
\end{remark}

We define the map  $h\colon\mathbb{R}^2\to\mathbb{R}^2$ by
\begin{equation}\label{difeomorphism_h}
h(\theta, \rho) = 
\begin{cases}
\big(\theta, \rho\big) & \text{if } (\theta, \rho) \notin R_*, \\
\big(\theta, \rho + \mu_{\theta_0}(\theta) \mu_{\rho_0}(\rho)\big) & \text{if } (\theta, \rho) \in R_*.
\end{cases}
\end{equation}

\begin{figure}[!htb]
\centering
\includegraphics[width=1\linewidth]{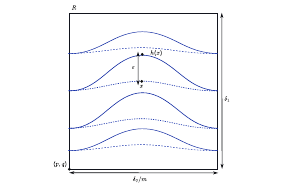}
\caption{\small Diffeomorphism $h$.}
\label{fig:h}
\end{figure}
We will show that $h$ is a diffeomorphism under certain conditions, for instance if $|\mu_{\theta_0}\mu'_{\rho_0}|<1$. Surjectivity follows from the following lemma.

\begin{lemma}\label{lemma:h_sobre}
Let $D\subset\mathbb{R}^2$ be a disk and let $h\colon D\to\mathbb{R}^2$  be a  continuous map that satisfies $h|_{\partial D}=Id$. Then $D\subset h(D)$.
\end{lemma}
\begin{proof}
Suppose, for contradiction, that there exists $p\in D$ such that 
$p\notin \image(h)$. Then the image of $h$ is contained in $\mathbb{R}^2\setminus\{p\}$.
Consider the boundary $\partial D\simeq \mathbb{S}^1$. 
Since $h|_{\partial D}=\mathrm{Id}$, we have $h(\partial D)=\partial D$. 
The inclusion of $\partial D$ into $\mathbb{R}^2\setminus\{p\}$ would be 
null-homotopic in $\mathbb{R}^2\setminus\{p\}$: there would exist a homotopy $\gamma_t$ from the constant 
loop $\gamma_0$ to $\gamma_1=\partial D$ inside $h(D)\subset\mathbb{R}^2\setminus\{p\}$, namely $\gamma_t(x)=h(tx)$.
But this is impossible because the loop $\partial D$ winds around the point $p$, 
so in $\mathbb{R}^2\setminus\{p\}$ it is not null-homotopic ($\mathbb{R}^2\setminus\{p\}$ is not simply connected). 
This contradiction shows that such a point $p$ cannot exist, and hence 
$D\subset h(D)$.
\end{proof}

To prove injectivity of $h$, suppose that $\rho\neq\widehat\rho$ and 
\[
\rho+\mu_{\theta_0}(\theta)\mu_{\rho_0}(\rho)=\widehat\rho+\mu_{\theta_0}(\theta)\mu_{\rho_0}(\widehat\rho).
\]  
Then, by the Mean Value Theorem, there exists $\tilde\rho \in [{\rho_0},{\rho_0}+\delta_1]$ such that  
\[
-\frac{1}{\mu_{\theta_0}(\theta)}=\frac{\mu_{\rho_0}(\widehat\rho)-\mu_{\rho_0}(\rho)}{\widehat\rho-\rho}=\mu'_{\rho_0}(\tilde\rho),
\]  
that contradicts the condition $|\mu_{\theta_0}\mu'_{\rho_0}|<1$. Hence, $h$ is bijective.

\noindent Now to establish that $h$ is a diffeomorphism, we compute its derivative:  
\begin{equation}\label{derivative_of_h}
Dh=\begin{pmatrix}
1 & 0 \\
\mu'_{\theta_0}(\theta)\mu_{\rho_0}(\rho) & 1+\mu_{\theta_0}(\theta)\mu'_{\rho_0}(\rho)
\end{pmatrix}
\end{equation}  
and note that the Jacobian determinant is $1+\mu_{\theta_0}(\theta)\mu'_{\rho_0}(\rho)\neq 0$. This shows that $h$ is a local diffeomorphism and injectivity ensures that $h$ is a diffeomorphism.

Next, we construct an important tool to locally perturb the stable and unstable manifolds inside a rectangle of base $\delta_0 > 0$ and height $\delta_1 > 0$, namely
\begin{equation}\label{rec}
[{\theta_0}, {\theta_0}+\delta_0] \times [{\rho_0}, {\rho_0}+\delta_1],
\end{equation}
with $({\theta_0},{\rho_0})\in\mathbb{R}^2$. This perturbation consists of the composition of several local perturbations supported on subrectangles into which $[{\theta_0}, {\theta_0}+\delta_0] \times [{\rho_0}, {\rho_0}+\delta_1]$ is divided, as in Figure~\ref{fig:curvature}. More precisely, for $m\in\N$ and for each $j = 1, \dots, m$, we define the subrectangles
\begin{equation}\label{spliting_rectangles}
R_j = [{\theta}_j, \, {\theta}_j + \delta_0/m] \times [{\rho_0}, \, {\rho_0}+\delta_1],
\end{equation}
where
$${\theta}_j={\theta_0}+(j-1)\delta_0/m,\  j\geq 1.$$

\begin{figure}[!htb]
\centering
\includegraphics[width=1\linewidth]{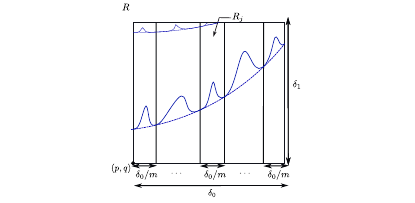}
\caption{\small Adding arbitrary curvature in the rectangle $R$.}
\label{fig:curvature}
\end{figure}

\begin{lemma}\label{lemma:curvatura}
For each $\varepsilon>0$, $K>0$, and $L>0$,
there is $m\in\N$ such that for each $j\in\{1,\dots,m\}$ there is $h_j\colon\mathbb{R}^2\to\mathbb{R}^2$ with $h_j-Id$ supported in $R_j$
%\emph{(\ref{spliting_rectangles})} 
satisfying: 
\begin{enumerate}
\item\label{lemma:item_i} 
if $\rho\colon[{\theta_0},{\theta_0}+\delta_0]\to[{\rho_0},{\rho_0}+\delta_1]$ of class $C^2$ satisfies $|\rho''|,|\rho'|<L$, 
then 
there is $t_{j}\in[{\theta}_j,{\theta}_j+\delta_0/m]$ such that
\begin{equation*}
\left\langle
%\frac{d^2}{dt^2}
\gamma''(t_{j}),N(t_{j})\right\rangle >K,
\end{equation*}
where $\gamma(\theta)=h_j(\theta,\rho(\theta))$,
%$\lambda=\grph(\rho)$,  
$N(t_{j})=(b,-a)$ and 
$\gamma'(t_{j})=(a,b)$;
\item\label{lemma:item_ii} $h_j$ is $\varepsilon$-close  to $Id$ in the $C^1$ topology.
\end{enumerate}
\end{lemma}

\begin{proof}
Part (\ref{lemma:item_i}). Let $m\in\N$ and note that for $j\in\{1, \cdots, m\}$ we have
\begin{eqnarray*}
&\gamma(\theta)=h_j(\theta,\rho(\theta))=\big(\theta,\rho(\theta)+\mu_{\theta_j}(\theta)\mu_{\rho_0}(\rho(\theta))\big),&\\
&\gamma'(\theta)=\big(1,\rho'(\theta)+\mu_{\theta_j}'(\theta)\mu_{\rho_0}\big(\rho(\theta)\big)+\mu_{\theta_j}(\theta)\mu'_{\rho_0}\big(\rho(\theta)\big)\rho'(\theta)\big),&\\
&\gamma''(\theta)=(0,\rho''(\theta)+\mu_{\theta_j}''\mu_{\rho_0}+2\mu'_{\theta_j}\mu'_{\rho_0}\rho'(\theta)+\mu_{\theta_j}\mu''_{\rho_0}(\rho'(\theta))^2+\mu_{\theta_j}\mu'_{\rho_0}\rho''(\theta)).&
\end{eqnarray*}
%in $R_j$ with $1\leq j\leq m$. 
Without loss of generality suppose that $||\gamma'(\theta)||=1$. Then, $\big\langle\gamma''(\theta),N(\theta)\big\rangle$ is given by 
\begin{equation*}
\rho''\!+m\frac{\delta}{\delta^2_0}\varphi''\!(t_\theta\!)\varphi\!(s_\theta\!)+\frac{2\delta}{\delta_0\delta_1}\varphi'\!(t_\theta)\varphi'\!(s_\theta)\rho'+\frac{\delta}{m\delta_1^2}\varphi(t_\theta)\varphi''(s_\theta)(\rho')^2\!+\frac{\delta	}{m\delta_1}\varphi(t_\theta)\varphi'(s_\theta)\rho''
\end{equation*}
with $\rho=\rho(\theta)$, and angular coordinates
$$t_\theta=m\left(\frac{\theta-{\theta_j}}{\delta_0}\right) \,\,\,\,\,\, \text{and} \,\,\,\,\,\, s_\theta=\frac{\rho(\theta)-{\rho_0}}{\delta_1}.$$ Consider a point \mbox{$\tilde\theta_j>0$} such that $\varphi'(t_{\tilde\theta_j})>0$ and $\varphi''(t_{\tilde\theta_j})>0$ (correspondent to the first inflection point of $\varphi'$, see Figure~\ref{fig:varphi}). Observe that fixed $\varepsilon, K,\delta_0,\delta_1$, and $\rho$, we have $$\big\langle\gamma''(\tilde\theta_j),N(\tilde\theta_j)\big\rangle\to+\infty, \,\,\,\,\,\, m\to+\infty.$$ Therefore, we may choose $m>0$ great enough so that
$\big\langle\gamma''(\tilde\theta_j),N(\tilde\theta_j)\big\rangle>K$, which concludes the proof of this part. 

Part (\ref{lemma:item_ii}). From the derivative of $h$ in (\ref{derivative_of_h}) and by (\ref{inequality_epsilon}),  we have  that
\begin{eqnarray*}
|\mu_{\theta_j}(\theta)\mu_{\rho_0}(\rho)|\leq\varepsilon/2, &|\mu'_{\theta_j}(\theta)\mu_{\rho_0}(\rho)|\leq\varepsilon/2, &|\mu_{\theta_j}(\theta)\mu'_{\rho_0}(\rho)|\leq\varepsilon/2.\\
\end{eqnarray*}
It can be verified that
\begin{eqnarray*}
\|h_j-Id\|\leq\varepsilon\\
\|Dh_j-Id\|\leq\varepsilon.
\end{eqnarray*}
Thus, $h\in B_\varepsilon(Id)$  which completes the proof.
\end{proof}

\subsection{Local Density of Cw-expansivity}
Consider the neighborhood $\mathcal{U}$ of $f_0$ given by the upper semicontinuity of $T$ for a neighborhood of $T(f_0)$.
We will perturb inside the annulus $\widetilde{A}$, in the rectangles expressed in polar coordinates as
\[
R = [\theta_0, \theta_0 + \delta_0] \times [\rho_0, \rho_0+\delta_1];
\]
$$R_j=[\theta_j,\theta_j+\delta_0/m]\times[\rho_0,\rho_0+\delta_1]$$
as in  Equation (\ref{spliting_rectangles}), see Figure~\ref{fig:rectangles}, with $r,\delta_0,\delta_1 > 0$ and $\theta_0\in\mathbb{R}$, $m>0$,  \mbox{$0\leq\theta_j-\theta_0\leq\delta_0$}, $1\leq j\leq m$. For $\nu>0$, let
\[
\begin{array}{cc}
D_0(\nu) = \bigl[\tfrac{\pi}{4}-\tfrac{\nu}{6}, \tfrac{\pi}{4}+\tfrac{\nu}{6}\bigr]\times \bigl[\tfrac{1}{8}, 8\bigr],
&
D_1(\nu) = \bigl[\tfrac{\pi}{4}-\tfrac{\nu}{3}, \tfrac{\pi}{4}+\tfrac{\nu}{3}\bigr]\times [1,2]
\\[6pt]
D_2(\nu) = \bigl[\tfrac{\pi}{4}-\tfrac{\nu}{6}, \tfrac{\pi}{4}+\tfrac{\nu}{6}\bigr]\times \bigl[\tfrac{3}{2},3\bigr],
&
D_3(\nu) = \bigl[\tfrac{\pi}{4}-\tfrac{\nu}{2}, \tfrac{\pi}{4}+\tfrac{\nu}{2}\bigr]\times \bigl[\tfrac{1}{8}, 8\bigr].
\end{array}
\]

\begin{figure}[!htb]
\centering
\includegraphics[width=1\linewidth]{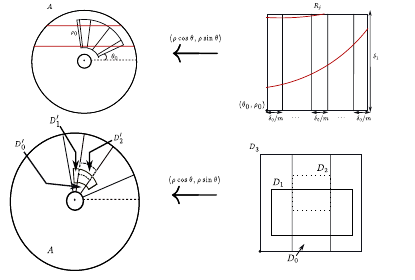}
\caption{\small Perturbation Rectangles.}
\label{fig:rectangles}
\end{figure}

When there is no risk of confusion, we omit the dependence on $\nu$ in the notation. 
Note that in $D_3$ the local stable manifolds of any $g$, can be written as graphs over the angular coordinate $\theta$, by the transversality between these and the \emph{radial lines}  $\theta=\const$, for $\nu$  sufficiently small.

\begin{remark}
%Let $\gamma:I\to M$ be a curve, such that $|\gamma'(t)|\neq0$ for all $t\in I$. 
In the case of diffeomorphism in $\mathcal{W}^1$, the stable and unstable manifolds are curves $\gamma:I=[\theta,\theta']\to M$ such that $|\gamma'(t)|\neq0$ for all $t\in I$. Therefore, we may suppose that these manifolds are parametrized by arc length, defined as $\ell(\gamma)(t)=\int_{\theta}^{t}\|\gamma'(\tau)\|d\tau$. 
\end{remark}

\begin{lemma}\label{lemma:diam_less_n}
Let $\gamma\colon I\to R$ be a curve of class $C^1$ that is a graphic in $R$. For $m>0$ sufficiently large, if $\ell(\gamma)>\frac{\delta_0}{m\cos\theta^*}$, then $$\gamma\not\subset R_j;$$
where $\theta^*=\max_{|t-\theta_j|\leq\delta_0/m}\{|\arccos\langle\frac{\gamma'(t)}{\|\gamma'(t)\|},i\rangle|; 1\leq j\leq m$ and $i=(1,0)\}$.
\end{lemma}
\begin{proof}
We suppose that $\gamma\subset R_j$ for some $j=1,\cdots,m$.
%and that $$\ell(\gamma)=\int_{\theta_j}^{t}\|\gamma'(\tau)\|d\tau>\frac{\delta_0}{m\cos\theta*},$$
%with $t\in[\theta_j,\theta_j+\delta_0/m]$. 
Since $\gamma$ is a graphic of the form $(\tau, \gamma_V(\tau))$, we have $\gamma'(\tau) = (1, \gamma_V'(\tau))$ and $\langle \gamma'(\tau), i \rangle = 1$. By definition of the cosine of the angle $\theta$ between $\gamma'(\tau)$ and $i$, we have
$$\cos \theta = \frac{\langle \gamma'(\tau), i \rangle}{\|\gamma'(\tau)\| \cdot \|i\|} = \frac{1}{\|\gamma'(\tau)\|},$$
hence $\|\gamma'(\tau)\| = \frac{1}{\cos \theta}$. Since %$\|\gamma'(\tau)\|\geq1$ and 
$\frac{1}{\|\gamma'(\tau)\|}=\cos\theta\geq\cos\theta*$, we have 
$$\ell(\gamma)=\int_{I}\|\gamma'(\tau)\|d\tau\leq\frac{1}{\cos\theta^*}\int_{I}d\tau\leq\frac{\delta_0}{m\cos\theta^*}.$$
This proves the contrapositive and concludes the proof.
\end{proof}
Consider the neighborhood $\mathcal{U}$ of $f_0$ in $\mathcal{W}^1$. By taking $\nu > 0$ sufficiently small and using the upper semicontinuity of $T$, we may assume that
\[
T(f) \subset \mathrm{int}(D_0) \quad \text{for all } f \in \mathcal{U},
\]
since $T(f_0)=\{(x_1,x_2)\in A \colon|x_2|=x_1\}$, and that the stable and unstable manifolds are graphics in $D_3$.

The following lemma ensures that, through a perturbation, one can obtain connected components of the dynamical ball with controlled diameter as in Figure~\ref{fig:h1}.

\begin{figure}[!htb]
\centering
\includegraphics[width=1\linewidth]{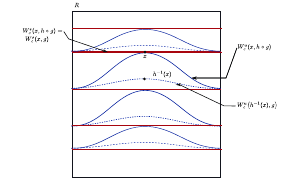}
\caption{\small The perturbation $h\circ g$ \emph{pushes up} the unstable manifolds of $g$ in $R$ while preserving the stable manifolds.}
\label{fig:h1}
\end{figure}

We will use the techniques developed for $h$, which is defined in $\R^2$, to perturb a diffeomorphism of the bitorus $M$. Formally, we should consider local charts as in \S\ref{secDefDA}, but we abuse the notation to keep it simpler.

\begin{lemma}\label{lemma:stable_unstable_perturbation}
Let $g\in\mathcal{U}$ and $R_j\subset D_3$ 
be a rectangle in polar coordinates contained in the annulus $\widetilde{A}$ such that any orbit that intersects $R_j$ does so at most once. If $h$ is a diffeomorphism as in (\ref{difeomorphism_h}) with $h=\mathrm{Id}$ outside $R_j$, then for all $x\in R_j$ we have
\begin{eqnarray*}
&W^s_c(x,h\circ g)\cap R_j=W^s_c(x,g)\cap R_j,&\\
&W^u_c(x,h\circ g)\cap R_j\subset h\Big(W^u_c\big(h^{-1}(x),g\big)\Big)\cap R_j.&
\end{eqnarray*}
\end{lemma}
\begin{proof}
For each $z\in W^s_\varepsilon(x,h\circ g)\cap R$ and $n\geq0$ 
%we will show that the stable manifold for the perturbation $h \circ g$ is the samea as for $g$. Therefore, for the $n$-th iterate and using the fact that the orbits of $g$ intersect $R$ at most once, we have
we have
$$(h\circ g)^n(z)=h\circ g\circ\cdots\circ h\circ g(z)=g^n(z)$$ 
since orbits of $g$ intersect $R$ at most once and $h=Id$ outside $R$.
This ensures that $$W^s_c(x,h\circ g)=W^s_c(x,g).$$
Now for each $z\in W^u_c(x,h\circ g)\cap R$ and $n\geq0$, we have
$$(h\circ g)^{-n}(z)=g^{-1}\circ h^{-1}\circ\cdots\circ g^{-1}\circ h^{-1}(z)=g^{-n}(h^{-1}(z))$$ 
since orbits of $g^{-1}$ intersect $R$ at most once and $h^{-1}=Id$ outside $R$. Thus,
$$\dist\!\big(g^{-n}(h^{-1}(z)),g^{-n}(h^{-1}(x))\big)=\dist\!\big((h\circ g)^{-n}(z)),(h\circ g)^{-n}(x))\big)\leq c,$$
ensuring that 
$h^{-1}(z)\in W^u_c\big(h^{-1}(x), g\big)$ and consequently that
$$W^u_c(x,h\circ g)\cap R\subset h\Big(W^u_c\big(h^{-1}(x), g\big)\Big)\cap R,$$ concluding the proof.
\end{proof}

Lemmas \ref{lemma:diam_less_n} and \ref{lemma:stable_unstable_perturbation} will be enough to ensure the dynamical ball $\Gamma_c(x)$ can be broken inside the rectangle and that each of its connected components are sufficiently small.

%\begin{lemma}[\cite{palismelo} Proposition 2.7]\label{c_2_dense_c_1}
%Let $M$ and $N$ be abstract manifolds with $M$ compact, $0\leq r<\infty.$
%Then 
%The subset of maps of class $C^{\infty}$ is dense in $C^r(M,N).$
%\end{lemma}

%\noindent The following theorem is a preliminary step for obtaining the residual set of $Cw$-expansive diffeomorphisms.

\begin{theorem}\label{theorem:F_n_dense}
The set $F_n$ is dense in $\mathcal{U}$ for all $n\in\N$.
%$$ \overline{F}_n = \mathcal{U}.$$
\end{theorem}
\begin{proof}
Let $g \in \mathcal{U}$, \( n\in\N \), and consider a ball $B_1(g,\varepsilon) \subset \mathcal{U}$ in $\diff^1(M)$. Choose a diffeomorphism $g_0 \in B_1(g, \varepsilon)$ of class $C^2$, see \cite{palismelo}*{Proposition 2.7}. This ensures that stable/unstable leaves of $g_0$ are also $C^2$ and we can consider their curvatures. The Invariant Manifold Theorem ensures that these curvatures are bounded (see \cite{takens}*{Theorem 5}), that is,
$$K:=\max_{x\in D_1}\{k(x); \,\,k(x) \text{ is the curvature of } W^s(x,g_0) \text{ at } x\}<\infty.$$
Take 
%If $T(g_0)=\emptyset$, then %$\widetilde{A}\subset\tpitchfork(g_0)$, so $g_0 \in F_n $ and %we are done. 
%Otherwise,
%it is possible to choose 
$m\in\N$  
so that the rectangle $D_1$ is  divided into finitely many subrectangles of the form
$$R_j=[\theta_j,\theta_j+\delta_0/m]\times[1,2],$$
with $\delta_0>0$ small, as in the Figure~\ref{fig:phases}, such that
$$\ell(W^s_c(x,g_0)\cap R_j)\leq\frac{\delta_0}{m\cos\theta^*}\leq\frac{1}{2n}$$
for all $x\in R_j$ and $j\in\{1,\dots,m\}$. This can be done since the stable leaves of $g_0$ in $D_1$ are $C^1$-close to the stable leaves of $f_0$, which are horizontal in $D_1$. We may assume that inside $D_3$ the stable and unstable manifolds are graphs of smooth functions $(\theta,\rho(\theta))$ over the polar coordinate which are $C^1$-close, so there is a common bound $L$ for their first and second derivatives. Without loss of generality we may suppose that $m$ also satisfies Lemma~\ref{lemma:curvatura}\eqref{lemma:item_i} for the numbers $\eps$, $K$, and $L$ as above.
   
%Now fix one rectangle $R_j$.
%as given above.  
It is possible to construct a diffeomorphism $h$ 
%such that $h-Id$ is supported in $R_j$, 
arbitrarily $C^1$-close to the identity supported in $D_1$ which is the composition of $m$ distinct $h_j$ given by Lemma \ref{lemma:curvatura} applied to each $R_j$. Also, for each $j\in\{1,\dots,m\}$ and each unstable manifold $\rho$ inside $R_j$ there is $\tilde\theta_j \in [\theta_j,\theta_j+\delta_0/m]$ so that $h\circ g_0$ pushes the unstable manifold upward ensuring that its curvature is larger than $K$ at $\tilde\theta_j$ and that if $(\tilde\theta_j(\rho),\rho)\in W^s(x,h\circ g_0)$ for some $x\in D_1$, then $(\tilde\theta_j(\rho),\rho)\in\tpitchfork(h_j\circ g_0)$.
%to break the dynamical ball. 
%For this purpose, take the curvature greater than  
%$$K=\max_{x\in D_1}\{k(x): k(x) \text{ curvature of } W^s(x,g_0) \text{ at } x\},$$ maximum curvature  $K<\infty$.
%in $\theta_{j,0}$.

Lemma~\ref{lemma:stable_unstable_perturbation} ensures that the stable manifolds of $h\circ g_0$ in $D_1$ equal the ones of $g_0$, while the unstable manifolds are pushed above to break the dynamical balls. Since $\Gamma^{h\circ g_0}_c(x)\subset W^s(x,g_0)$,
%,  then we have $$\Gamma_c^{h\circ g_0}(x)\subset \Gamma_c^{g_0}(x),$$ for all $x\in M$, because any orbit that intersects $R_j$ does so at most once.
Lemma~\ref{lemma:diam_less_n} ensures that the connected components of $\Gamma^{h\circ g_0}_c(x)$ inside $D_1$ satisfy
$$\diam(\com_{x}\Gamma^{h\circ g_0}_c(x)\cap R_j)\leq \ell(W^s_c(x,g_0)\cap R_j)\leq\frac{\delta_0}{m\cos\theta^*}\leq\frac{1}{2n}$$ for all $x\in R_j$ and $j\in\{1,\dots,m\}$, and also
%$R_j$  yields
%$$\diam(\com_{x}\Gamma^{h\circ g_0}_c(x)\cap R_j)<\frac{1}{2n}$$
%for all $x\in R_j,j=1,\dots,m$ and
$$\diam(\com_{x}\Gamma^{h\circ g_0}_c(x))<1/n$$
for all $x\in D_1$ since in $D_1$ there are no long arcs contained in dynamical balls in three consecutive $R_j$.

Finally, to eliminate possible long interval of tangencies  along the boundary $g_0(\mathbb{S}^1\cap D_0)$, we apply the same perturbation procedure inside the rectangle $D_2$. After these local modifications, we obtain a diffeomorphism
$$h\circ g_0\in B_1(g,\varepsilon)\cap F_n.$$
Since $g\in\mathcal{U}$ and $\varepsilon>0$ were arbitrary, we conclude that $F_n$ is dense in $\mathcal{U}$.
\end{proof}

\begin{figure}[!htb]
\centering
\includegraphics[width=1\linewidth]{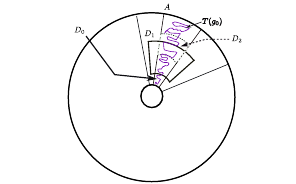}
\caption{\small Perturbations in the rectangles $D_1$ and $D_2$.}
\label{fig:phases}
\end{figure}

Finally, the proof of Theorem \ref{theorem:main_cw_exp1} goes as expected.

\begin{proof}[Proof of Theorem \ref{theorem:main_cw_exp1}]
Let \( \mathcal{R} = \bigcap_{n\geq 1} F_n \), where each \( F_n \) is defined as in (\ref{definition_F_n}). Since each $F_n$ is an open (Lemma \ref{open}) and dense (Theorem \ref{theorem:F_n_dense}) subset of $\mathcal{U}$, it follows that \( \mathcal{R}\) is a residual subset. Each element of \( \mathcal{R} \) is uniformly cw-expansive (Proposition \ref{proposition:residual}), with a common constant \( c > 0 \), and since $\Omega(g)\neq M$ for $g\in\mathcal{R}$, it follows that \( g \) is not expansive.
%\( \diff^1(M) \) is a Baire space, it follows that \( \mathcal{R}\) is a dense subset of \( \mathcal{U} \), 
Hence, every diffeomorphism in \( \mathcal{R} \) is cw-expansive but not expansive. 
\end{proof}

\begin{remark}
Since \( \diff^1(M) \) is a Baire space, this Theorem implies that 
%each set \( F_n \) is dense in \( \mathcal{U} \), and that 
the set of cw-expansive diffeomorphisms \( E_{cw}\cap\mathcal{U} \subset \mathcal{R} \) is dense in \( \mathcal{U} \).
\end{remark}

\vspace{+0.4cm}

\hspace{-0.45cm}\textbf{Acknowledgments.}
Alfonso Artigue was supported by ANII and PEDECIBA at Uruguay. Bernardo Carvalho was supported by CNPq project number 446192/2024. José Cueto was supported by CNPq grant number 317565/2025-7. This article is part of the Ph.D. thesis of José Cueto, under the supervision of the other authors, defended at the Federal University of Minas Gerais in November/2025.

\vspace{+0.4cm}

\hspace{-0.45cm}\textbf{Added in proof.} A recent work by Ziteng Ma \cite{ziteng} has shown that, generically, $n$-expansive surface diffeomorphisms must be Anosov. This article, which appeared online on November 17 after the defense of the above mentioned thesis, complements our results.

\vspace{1.0cm}
\noindent

{\em A. Artigue}
\vspace{0.2cm}

\noindent

Departamento de Matem\'atica y Estadística del Litoral,

Universidad de la República,

Gral. Rivera 1350, Salto, Uruguay
\vspace{0.2cm}

\email{artigue@unorte.edu.uy}

\vspace{1.0cm}
\noindent

{\em B. Carvalho}
\vspace{0.2cm}

\noindent

National Laboratory for Scientific Computing – LNCC/MCTI 

Av. Getúlio Vargas 333, CEP 25651-070, 

Petrópolis – RJ, Brazil

\vspace{0.2cm}

\email{bmcarvalho@lncc.br}

\vspace{1.0cm}
{\em J. Cueto}
\vspace{0.2cm}

\noindent

Departamento de Matem\'atica,

Universidade Federal de Minas Gerais - UFMG

Av. Ant\^onio Carlos, 6627 - Campus Pampulha

Belo Horizonte - MG, Brazil.

\vspace{0.2cm}

\noindent

National Laboratory for Scientific Computing – LNCC/MCTI 

Av. Getúlio Vargas 333, CEP 25651-070, 

Petrópolis – RJ, Brazil

\email{joscueto@lncc.br}

\end{document}